 \documentclass{amsart}
\usepackage{amsmath,amssymb,amsthm,textcomp,enumerate,multicol,graphicx,mathtools,subfloat,xcolor,tikz,tikz-cd,hyperref, comment, eucal, units, framed, pdfpages,soul, yfonts, caption, subcaption, tkz-fct, url, booktabs, amsfonts, nicefrac, microtype, enumitem, doi, hyperref, svg}
\usepackage{thm-restate}
\graphicspath{ {./images/} }
\usepackage[square,numbers]{natbib}
\usepackage[headings]{fullpage}
\usetikzlibrary{calc,through,intersections, arrows, decorations.markings, matrix, shapes, snakes, decorations, braids, arrows.meta}

\tikzset{mynode/.style={draw, very thick, circle, minimum size=1cm},
    myarrow/.style={very thick, -Triangle}}
\usepackage[utf8]{inputenc} 
\usepackage[T1]{fontenc}    
\usepackage[all]{xy}
\usepackage{xcolor,cancel}
 \setul{0.5ex}{0.3ex}
    \definecolor{Green}{rgb}{0,1,0}
    \setulcolor{Green}
    
    \setul{0.5ex}{0.3ex}
    \definecolor{Red}{rgb}{1,0.0,0.0}
    \setulcolor{Red}

    \setul{0.5ex}{0.3ex}
    \definecolor{Blue}{rgb}{0,0.0,1}
    \setulcolor{Blue}
\usepackage{cleveref}

\usepackage{subcaption} 

\usepackage[margin=1in]{geometry}

\theoremstyle{plain}
\newtheorem{theorem}{Theorem}
\numberwithin{theorem}{section}
\newtheorem{lemma}[theorem]{Lemma}
\newtheorem{proposition}[theorem]{Proposition}
\newtheorem{corollary}[theorem]{Corollary}
\newtheorem{conjecture}[theorem]{Conjecture}
\theoremstyle{definition}

\newtheorem{remark}[theorem]{Remark}

\newtheorem{question}[theorem]{Question}
\newtheorem{example}[theorem]{Example}
   
\newcommand{\ra}{\rightarrow}
\newcommand{\bdry}{\partial}
\newcommand{\Scal}{\mathcal{S}}

\newcommand{\R}{\mathbb{R}}
\newcommand{\e}{\varepsilon}
\title{Generating Infinitely Many Hyperbolic Knots with Plats}
\author{ Carolyn Engelhardt and Seth Hovland }

\hypersetup{
pdftitle={Generating Infinitely Many Hyperbolic Knots with Plats},
}

\begin{document}
\maketitle

\begin{abstract} 
\noindent In this paper we study the relationships between links in plat position, the dynamics of the braid group, and Heegaard splittings of double branched covers of $S^3$ over a link.  These relationships offer new ways to view links in plat position and a new tool kit for analyzing links.  In particular, we show that the Hempel distance of the Heegaard splitting of the double branched cover obtained from a plat is a lower bound for the Hempel distance of that plat.  Using the  Hempel distance of a knot in bridge position and pseudo-Anosov braids we obtain our main result: a construction of infinitely many sequences of prime hyperbolic $n$-bridge knots for $n \geq 3$, infinitely many of which are distinct. We consider known results to show that the knot genus and hyperbolic volume of these knots are bounded below by a linear function.    
\end{abstract}

{\bf Key Words: Hempel Distance, Braid Groups, Plat Closure, pseudo-Anosov braids} 

\section{Introduction}\label{sec:introduction}
Until recently, it was widely believed that generic prime links are hyperbolic.  Indeed, this holds for prime knots of small crossing number \cite{smalllinks}.  However, Malyutin proved in \cite{malyutin} that, counter-intuitively, the proportion of hyperbolic links as crossing number increases does \emph{not} approach 1.  This leads to the question: How can one construct prime hyperbolic knots of high crossing number?  This question has been approached, for instance, in \cite{hyperbolic_modifications} by modifying hyperbolic links with at least $2$ components combinatorially. 

In this paper, we approach this question by studying knots as \emph{plat closures} of elements of the braid group (called \emph{plats}).  By Schubert's work in \cite{Schubert1954berEN}, every link in $S^3$ admits a bridge position for some $n$, and therefore can be represented as the plat closure of a word in the braid group $B_{2n}$. 
Links in plat position admit a measure called the \emph{Hempel distance}, which is an analog of the Hempel distance for Heegaard splittings of $3$--manifolds \cite{HEMPEL2001}.  Bachman and Schliemer showed in \cite{Bachman2003DistanceAB} that if the Hempel distance of a knot in bridge position is greater than $3$, then the knot is hyperbolic.  This gives us a criterion to look for when constructing hyperbolic knots.  


Representing a knot $K$ the plat closure of the braid word $\beta$ allows us to explicitly construct a Heegaard splitting of $Y(K)$, the double branched cover of $S^3$ over $K$.  Using this construction, we can give a lower bound on the Hempel distance of a knot. 

\begin{restatable}[]{prop}{knotdistanceboundedbelow}
\label{introprop:knot_distance_bounded_below}
If $(K, S)$ is a bridge splitting of a knot in $S^3$ and $H=H_+ \cup_\phi H_-$ is the Heegaard splitting for the double cover of $S^3$ over $K$, $Y(K)$.  Then $d(Y, H) \leq d(K, S)$.
\end{restatable}


\noindent In order to show a knot $K$ is hyperbolic, we start by finding the Hempel distance of $Y(K)$.  
It was shown in \cite{DistancesofHeegaardsplittings} for pseudo-Ansov maps $\varphi$ on surfaces with genus $g \geq 2$, if $\varphi$ meets certain criteria, then the Hempel distances of the Heegaard splitting with splitting maps $\varphi^m$ grow as a linear function of $m$.

Using Proposition~\ref{introprop:knot_distance_bounded_below}, we can then construct a family of $n$--bridge hyperbolic knots whose Hempel distance is bounded below by a linear function in $m$ by fixing a pseudo-Anosov braid word $\beta$ which meets certain technical conditions and taking the plat closures of $\beta^m$.  

\begin{restatable}[]{thm}{infhypprimeknots}
\label{introthm:infhypprimeknots}
 For $n\geq 3,$ if $\beta\in B_{2n}$ is a generic, pseudo-Anosov braid with $1$-component plat closure, then plat closures of certain powers of $\beta$ will generate a sequence of infinitely many distinct prime hyperbolic knots whose volumes are strictly increasing.  
\end{restatable}

In particular, if $\beta$ is chosen to have only positive crossings, the crossing numbers of this family of hyperbolic knots also increases.  This provides a method for generating high crossing number hyperbolic knots.  
Notably, the Hempel distance of $k$ in such a family is bounded \emph{below} by a linear function in $m$, but we do not have an upper bound for the Hempel distance. In contrast, in \cite{bridgedistanceplat} Moriah and Johnson calculated precisely the Hempel distance for plat closures of braid words of a certain form, called \emph{highly-twisted} knots.  Our construction does not require that the plat projections are highly-twisted, as in Examples~\ref{example:3-bridgeconcreteexample} and~\ref{example:3_-1} (though we do not know if these admit highly-twisted diagrams in the sense of \cite{HighlyTwisted}).  That is, our construction is slightly more general, but loses the sharpness of 
\cite{bridgedistanceplat}.  However, there is clear synergy; in Example~\ref{example:highlytwisted}, we construct a pseudo-Anosov braid whose plat closure is a highly-twisted knot and use the results of Moriah and Johnson to explicitly calculate the Hempel distances of the plat closures of powers of this braid word.


Also of interest, our family of knots have unbounded volume. It has been shown by Bachman and Schleimer in \cite{Bachman2003DistanceAB} that the genus of a knot is bounded below by a linear function in the Hempel distance of the knot (see Corollary \ref{cor:knot_genus_Hempel_distance}).  The genus of a hyperbolic knot is known to be linearly related to the volume of the knot \cite{brittenham1998boundingcanonicalgenusbounds}.  Thus, the volume of a hyperbolic knot is bounded below by a  linear function in the Hempel distance.  A key feature of our sequence of plat closures is that the Hempel distance increases with each power, so that the volume of the hyperbolic knots increases as well.

\vspace{1em}
The paper is outlined as follows:
in Section 2 we introduce links in plat position and discuss the Nielsen-Thurston classification of braids.  In particular, we consider pseudo-Anosov braids.  Then we discuss the Hempel distance for both Heegaard splittings and links in bridge position and recall some facts about how the distance of a link in bridge position relates to its link type.  In Section 3 we recall the method for lifting a link in plat position to obtain a 3-manifold and a Heegaard splitting that corresponds to the link.  We show how to ``see'' the gluing map of the lift by considering a \emph{shadow diagram} of the plat.  We then prove our main proposition, the Hempel distance of a link is bounded below by the Hempel distance of the splitting map of the lift.  In Section 4, prove  our main theorem.  We show an explicit sequence that generates infinitely many hyperbolic knots whose genus and hyperbolic volume are bounded below by linear functions.  Our final section provides some remaining questions and a few ideas for further research.

\subsection*{Acknowledgements}\label{subsec:acknowledgements}

We thank William Menasco and \c Ca\u gatay Kutluhan for their constant guidance and suggestions on this research, as well as Deepisha Solanki, Greg Vinal, and Rom\'an Aranda for taking an active interest in this research and for many insightful conversations.  We would also like to thank Dan Margalit and Saul Schleimer for their suggestions on early drafts of this manuscript.  We thank for Ryan Blair, Nathan Dunfield, Nir Lazarovich, Eric Samperton, and Moirah Yoav for their helpful correspondence and insights into the first draft of this paper.  Carolyn Engelhardt was supported in part by a Simons Foundation grant No. 519352.


\section{Preliminaries} 

Throughout this paper we restrict our attention to links in $S^3$ and use $\overline{X}$ to denote a manifold $X$ with reversed orientation.  In this section, we review the definitions and notation for plats and Hempel distance.  We also review known results about the Hempel distance of knots which we will use in the construction in Section~\ref{sec:Generating_a_Sequence_of_Hyperbolic_Knots}.

\subsection{Plat Closures and Pseudo-Anosov Braids}\label{subsec:PlatClosuresandPseudo-AnosovBraids}

\subsubsection{$B_{2n}, B(S^2)_{2n},$ and Plats}\label{subsubsec:plat_bridge_braid}

Let $B_{2n}$ denote the braid group on $2n$-strands and $B(S^2)_{2n}$ denote the spherical braid group on $2n$ strands.  The \emph{plat closure} of a braid word $\beta$ in $B_{2n}$ (respectively, $B(S^2)_{2n})$, denoted as $\widehat{\beta}$, is the link in $\R^3$ (respectively, $S^3$) obtained by joining the $(2i-1)^{th}$ top strand to the $2i^{th}$ top strand and the $(2i-1)^{th}$ bottom strand to the $2i^{th}$ bottom strand with untwisted arcs for $1\leq i\leq n$.  We call $\widehat{\beta}$ a \emph{plat}.  See Figure~\ref{fig:platclosureofbraidword} for an illustration.  We note that plat closures are only defined for braid groups with an even number of strands.  
\begin{figure}[ht]
    \centering
    \includegraphics[width=0.5\linewidth]{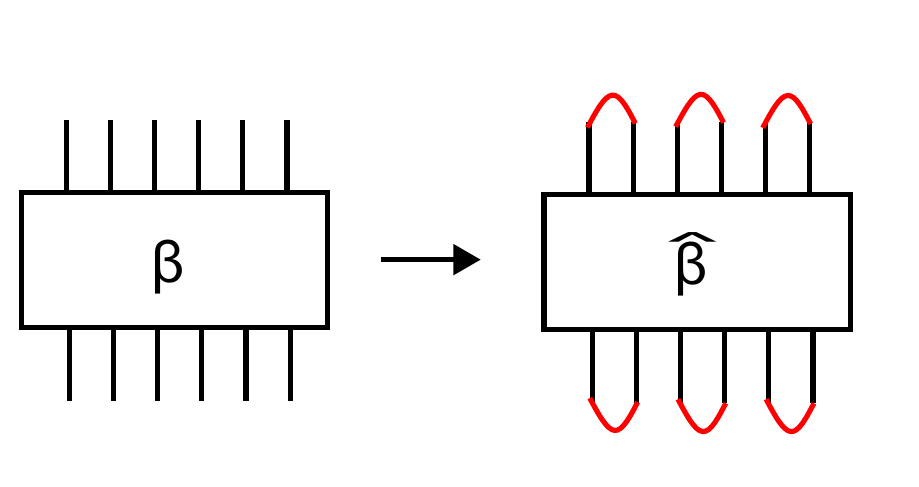}
    \caption{The \emph{plat closure} of $\beta.$}
    \label{fig:platclosureofbraidword}
\end{figure}

By \cite{Schubert1954berEN}, every link $L$ in $S^3$ admits an \emph{$n$--bridge decomposition} with respect to some splitting sphere $S^2$ for some $n \geq 1$ (see Section~\ref{subsec:Hempel_Distance_of_Links}).  An $n$--bridge splitting of $(S^3, L)$ corresponds to a Morse function $f \colon S^3 \ra \R$ such that $f|_L$ has $n$ maxima, all of which are above $S^2$, and $n$ minima, all of which are below $S^2$.  By perturbing this Morse function to be self-indexing, we see that every link $L$ in $S^3$ is the plat closure of some braid word $\beta \in B(S^2)_{2n}$ for some $n$.  By treating $S^3$ as $\R^3 \cup \{\infty\}$, where $\infty$ is a point on $S^2$ away from $L$, we can treat $L$ as the plat closure of $\beta \in B_{2n}$.  
The representation of a link as a plat is far from unique.  In \cite{Birman_stable_eq}, Birman proved that if $\widehat{\beta}, \widehat{\beta^\prime}$ are the same link, then $\beta \in B(S^2)_{2n}, \beta^\prime \in B(S^2)_{2m}$ are related by a sequence of \emph{stabilizations, destabilizations}, and left and right multiplication by words in $H_{2n}$, Hilden's subgroup \cite{Hilden_two_groups} (and similarly for $\beta, \beta^\prime \in B_{2n})$.  An example of stabilization is illustrated in Figure~\ref{fig:bridgestabilization}.  

There are distinctions between plats in $B_{2n}$ and in $B(S^2)_{2n}$.  For example, there is one Hilden double coset of the unknot in $B(S^2)_{2n}$ \cite{solanki_plats_unlink}, but this is unknown for $B_{2n}$.  However, for our purposes, we can pivot between considering $\beta \in B_{2n}$ or $B(S^2)_{2n}$ as needed, as discussed in Section~\ref{subsubsec:Pseudo-Anosov_Braids}.  We will generally discuss $\beta \in B_{2n}$.  

\subsubsection{Number of Components in a Plat Closure}\label{subsubsec:number_of_components}

For $\beta\in B_{2n}$, $\widehat{\beta}$ will be a link of up to $n$ components. We seek to study \emph{knots}, as opposed to links with multiple components.  In order to do so, we must first describe a method of determining how many components a plat $\widehat{\beta}$ has.  Consider $B_{2n}$ as the mapping class group of the $2n$-punctured disk. The map $\pi \colon B_{2n} \ra S_{2n}$, called the \emph{canonical projection} from the braid group to the symmetric group, is a surjective homomorphism which takes $\beta$ to the permutation which sends the punctures $(1, 2, ..., 2n)$ to their images under $\beta$.   See Figure~\ref{fig:braidpermutation}. 

\begin{figure}[ht]
    \centering
    \includegraphics[scale=0.4]{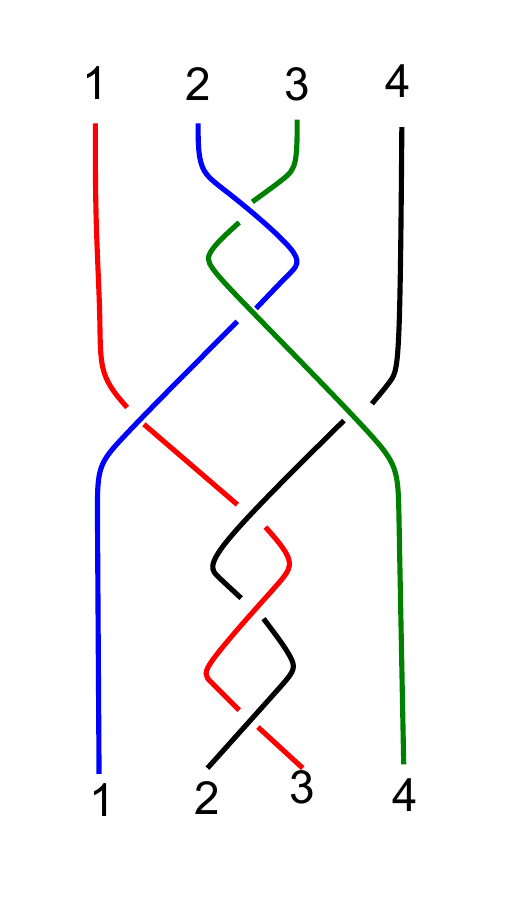}
    \caption{The braid $\sigma_2^2\sigma_1^{-1}\sigma_3\sigma_2^{-3}$.  We see that $\pi(\sigma_2^2\sigma_1^{-1}\sigma_3\sigma_2^{-3}) = [3,1,4,2]$.}
    \label{fig:braidpermutation}
\end{figure}

Since the plat closure of $\beta$ is obtained by joining the $(2i-1)^{th}$ top strand to the $2i^{th}$ top strand and the $(2i-1)^{th}$ bottom strand to the $2i^{th}$ bottom strand for $1\leq i\leq n,$ the permutation $\pi(\beta)$ allows us to determine the number of components of $\widehat{\beta}$ by constructing a graph, called the \emph{plat closure graph} of $\beta$, as follows: there are $2n$ vertices, labelled $1, ..., 2n$.  For $1 \leq i \leq n$, there is an edge between vertices $2i-1$ and $2i$ (these edges correspond to the top bridges of the plat closure).  Write the permutation $\pi(\beta)$ as an $n$--tuple with $i^{th}$ entry the image of the $i^{th}$ puncture under $\beta$.  Then, for $1 \leq i \leq n$, add an edge between the vertices in the $(2i-1)^{th}$ and $2i^{th}$ entries of $\pi(\beta)$. 

\begin{proposition}\label{prop:plat_closure_graph_components}
    Let $\beta\in B_{2n}$.  Then the number of components of $\widehat{\beta}$ is exactly the number of connected components of the plat closure graph of $\beta$. 
\end{proposition}

\begin{proof}
    By construction, the edges between $2i-1$ and $2i$ are the top bridges of $\widehat{\beta}$ and the edges between the vertices in the $(2i-1)^{th}$ and $2i^{th}$ entries of $\pi(\beta)$ are the bottom bridges of $\widehat{\beta}$.  The attachment of these bridges to the strands of $\beta$ determines a handle decomposition of $\widehat{\beta}$, which is sufficient to detect the number of components of $\widehat{\beta}$. 
\end{proof}

We note that if the plat closure of a braid $\beta$ is a knot, that does \emph{not} mean that all powers of $\beta$ have $1$-component plat closures. For example, the plat closure of the braid $\beta=\sigma_1\sigma_2\sigma_3\in B_{4}$ is an unknot.  However, $\beta^4=\Delta^2$, where $\Delta$ is the Garside element, so $\widehat{\beta^4}$ is a 4-component unlink. 

\vspace{1em}

\begin{example}\label{example:plat_graph}
    Suppose we have a relative complicated braid word 
    $$\beta=\sigma_4^{-1}\sigma_2\sigma_3^{-1}\sigma_2^{-1}\sigma_3\sigma_4^{-2}\sigma_3^{-1}\sigma_2\sigma_3^{-1}\sigma_2^{-1}\sigma_5^{-1}\sigma_4$$ 
    We see that $\pi(\beta)=[1,6,5,2,3,4].$  Our plat closure graph corresponding to $\beta$ is as in Figure~\ref{fig:platclosuregraph}.  From this graph, we see that $\widehat{\beta}$ has two components, one of which is $1$-bridge and the other of which is $2$-bridge.
    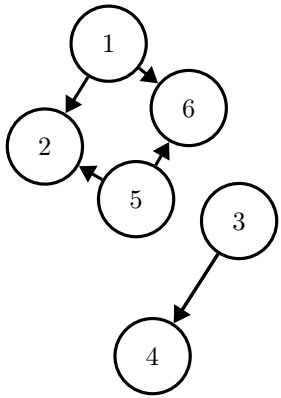
\begin{figure}[ht]
        \centering
        \begin{tikzpicture}[scale=0.7]
        \node[mynode](n5) at (0,0){5};
        \node[mynode](n6) at (60:2){6};
        \node[mynode](n3) at (168:-2){3};
        \node[mynode](n2) at (150:2){2};
        \node[mynode](n4) at (276:3){4};
        \node[mynode](n1) at (100:3){1};
        \draw[myarrow](n1)--(n2);
        \draw[myarrow](n3)--(n4);
        \draw[myarrow](n5)--(n6);
        \draw[myarrow](n1)--(n6);
        \draw[myarrow](n5)--(n2);
        \draw[myarrow](n3)--(n4);
        \end{tikzpicture}
        \caption{Plat Closure Graph of $\beta$}
        \label{fig:platclosuregraph}
    \end{figure}
\end{example}

\subsubsection{Pseudo-Anosov Braids}\label{subsubsec:Pseudo-Anosov_Braids}

The spherical braid group $B(S^2)_{n}$ is isomorphic to the mapping class group of $S^2$ with $n$ punctures, and the braid group $B_{n}$ is isomorphic to the mapping class group of a disk $D^2$ with $n$ punctures \cite{Birman_MCG_Bn}.  By the Nielsen-Thurston classification of elements of the mapping class group of a surface \cite{NielsenThurston}, we have that each braid word $\beta \in B(S^2)_{n}$ or $\beta\in B_{n}$ may be classified as either: 
\begin{enumerate}
    \item \emph{Periodic:} there exist positive integers $s,t$ with $\beta^s=\Delta^t$ where $\Delta$ is the Garside element of $B_{n}.$
    \item \emph{Reducible:}  there exists a collection of disjoint curves on the punctured sphere fixed by $\beta.$
    \item \emph{Pseudo-Anosov:} powers of $\beta$ stabilize (leave invariant) a pair of foliations on the punctured sphere, called stable and unstable.  A transverse measure on each of these foliations is expanded or contracted by the dilation, $\lambda$. 
\end{enumerate}
\noindent See \cite{BraidsandDynamics} for an introduction to the subject. 

The Artin presentation of $B(S^2)_{n}$ has the same generators as the Artin presentation of $B_{n}$ and the same relations, with the addition of the relation
    $\sigma_1\sigma_2\cdots\sigma_{n-1}^2\sigma_{n-2}\cdots\sigma_1=1$.  
\noindent If $\iota \colon B_{2n} \ra B(S^2)_{n}$ is the natural map sending $\sigma_i$ to $\sigma_i$, then $\iota(\beta)$ has the same Nielsen-Thurston classification as $\beta$.  Because our primary concern about braids will be their Nielsen-Thurston classification, we can consider $\beta \in B_{2n}$, as opposed to working explicitly in $B(S^2)_{2n}$.

In this work, we are primarily interested in the plat closures of pseudo-Anosov braids.  By Theorem 5.1 of \cite{pAbraidsareGeneric}, pseudo-Anosov braids are generic in $B_{n}$ in the following sense: if we consider all braid words in $B_n$ of length at most $\ell$, the proportion of pseudo-Anosov braids in this collection increases exponentially as $\ell$ increases.  

Pseudo-Anosov braids have many interesting properties; in particular, they leave invariant a foliation on the punctured disk. The unstable foliation in turn gives rise to an invariant \emph{train track}, which is an immersed 1-manifold. See Section~\ref{subsubsec:Train-Tracks_Method_for_Standard_Pseudo-Anosov_Braids} for further discussion and Example~\ref{example:3-bridgeconcreteexample} for an example of a train track on a disk. 
Also of note: pseudo-Anosov braids admit a numerical measure called \emph{entropy} which is the logarithm of the dilation $\lambda$.  Heuristically, one can think of braids as ways to ``mix up'' the disk, and pseudo-Anosov braids as ways to ``mix up'' the disk as much as possible.  The entropy of a pseudo-Anosov braid can then be thought of as measuring how quickly the disk is mixing.



\subsection{Hempel Distance of Heegaard Splittings}\label{subsec:Hempel_Distance_of_Heegaard_Splittings}
A genus $g$ \emph{Heegaard splitting} of a $3$-manifold $Y$ is a decomposition $Y=H_\alpha \cup_\phi \overline{H_\beta}$, where $H_\alpha$ is a handlebody which is defined by a set of $g$ pairwise disjoint nonseparating simple closed curves called $\alpha = \{\alpha_1, ..., \alpha_g\}$ (and similarly, $H_\beta$ is defined by $\beta = \{\beta_1, ..., \beta_g\}$).  The map $\phi \colon F_g \ra F_g$ 
is called the \emph{splitting map}; this is the map which glues $H_\alpha$ to $H_\beta$ along $F_g = \bdry H_\alpha = \bdry H_\beta$.

Given a closed genus $g$ surface $F$, the \emph{curve graph} of $F$, denoted $\mathcal{C}(F)$, is the graph whose vertices are proper isotopy classes of essential simple closed curves on $F$.  There is an edge between two vertices if and only if there are representatives of the corresponding isotopy classes which are disjoint.  We encourage an interested reader to see \cite{MM1999} for more details.  We note that the curve graph is a special case of the arc and curve graph of a surface, discussed below in Section~\ref{subsec:Hempel_Distance_of_Links}, since a closed surface does not contain any properly embedded arcs.  

In \cite{HEMPEL2001} Hempel introduced a measure of complexity for Heegaard splittings of $3$-manifolds now called the \emph{Hempel distance} of the splitting, which is the minimal distance in $\mathcal{C}(F)$ between a curve $A$ which bounds a disk in $H_\alpha$ and a curve $B$ which bounds a disk in $H_\beta$.  
%
We denote the Hempel distance of a Heegaard splitting $H$ of a $3$-manifold $Y$ as $d(Y,H)$. 

Also in \cite{HEMPEL2001}, Hempel showed that by taking powers of certain pseudo-Anosov elements in the mapping class group of a splitting surface $F$, one can construct Heegaard splittings of $3$-manifolds with arbitrarily high distance by using these elements as the splitting map.  As for the effect on Hempel distance of the reducible and periodic mapping class group elements, not much work can be done.  If a reducible element in the mapping class group fixes any curve in the splitting surface which bounds a disk in either handlebody, then the Hempel distance of the splitting is 0.  A similar argument bounds the Hempel distance of Heegaard splittings with periodic mapping class group elements as their splitting maps.

\subsection{Hempel Distance of Links}\label{subsec:Hempel_Distance_of_Links}

 Throughout this paper, we will view $S^3$ as the standard genus 0 Heegaard splitting $S^3=B^3\cup_{S^2} \overline{B^3}$.  A link $L$ is in $n$--\emph{bridge position} if $L$ is positioned such that it intersects the 2-sphere $S^2$ transversely in $2n$ points, and in each $B^3$ handlebody, $L$ is a collection of $n$ properly embedded, pairwise disjoint, simultaneously boundary parallel arcs.  We call such a collection a \emph{trivial tangle} (or often, just a \emph{tangle}).   
There are many different choices for the $S^2$ splitting surface of the genus $0$ Heegaard splitting of $S^3$ and $L$ may be in bridge position with a different choice of $S^2.$  Thus, when we consider a link in bridge position we will always mean with respect to a fixed $S^2$ splitting surface, which we denote as $(L, S^2)$.  This is also called a \emph{bridge decomposition} of $(S^3, L)$. It is clear that the plat closure of a braid is in bridge position.  Thus, all the statements below about the Hempel distance of a link in bridge position holds for the plat closures of braids.  In a forthcoming paper, the second author proves that there is a bijective correspondence between the plat closures of braids $\beta\in B_{2n}$ and links in bridge position up to bridge isotopy.  That is, $L$ is in $n$-bridge position if and only if $L$ is isotopic to the plat closure of some $\beta \in B_{2n}$ via an isotopy which preserves bridge position.

We will treat a tangle $T_\beta \subset B^3$ as the \emph{upper-plat closure} of a braid word.  We use $T_\e$ to denote the upper plat closure of the empty word $\e$.  We use $\overline{T_\beta}$ to denote $T_\beta$ with reversed orientation; that is, $\overline{T_\beta} \subset \overline{B^3}$, the lower plat closure of $\beta^{-1}$.   
\begin{figure}[ht]
    \centering
    \includegraphics[scale=0.6]{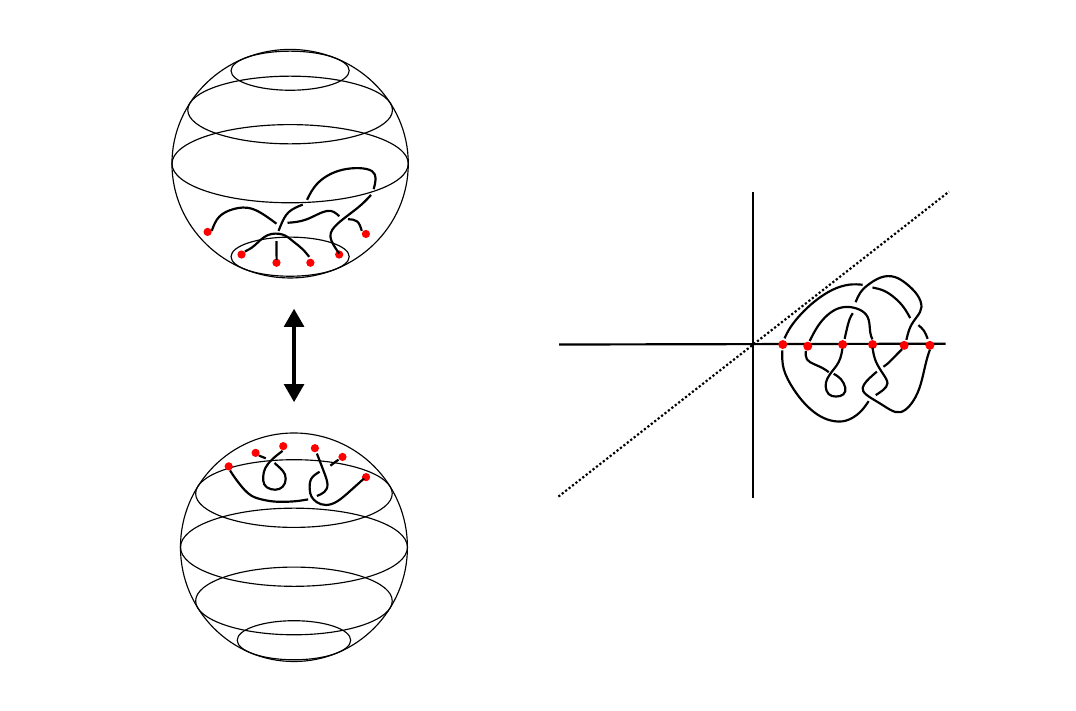}
    \caption{Two ways to view a knot in bridge position. The image on the left shows the Heegaard splitting of $S^3$. In the image on the right, the splitting sphere is thought of as the $xy$-plane compactified with the point at $\infty$.}
    \label{fig:twoviewsofbridgepositin}
\end{figure}

Given a link $L$ in $n$--bridge position, it punctures the splitting sphere $S^2$ in $2n$ points. We will denote this $2n$-punctured sphere by $X_{2n}.$  With some abuse of notation, we will occasionally also call $X_{2n}$ the \emph{splitting sphere} for the genus-$0$ Heegaard splitting of $S^3$. Corresponding to $X_{2n}$, we construct a graph called the $\emph{Arc and Curve Graph}$, denoted $\mathcal{AC}(X_{2n}).$  The vertices of $\mathcal{AC}(X_{2n})$ are proper isotopy classes of essential arcs and curves on $X_{2n}.$  There is an edge between two vertices if and only if there are representatives of the corresponding isotopy classes which are disjoint.  We encourage an interested reader to see \cite{MM1999} for more details.

With some abuse of notation, let $S^3 \backslash L$ denote $S^3 \backslash \text{int}({N(L)})$, the \emph{link exterior}.  Notice that $X_{2n}$ divides $S^3\backslash L$ into two 3-balls with $n$-arcs removed.  We call these submanifolds $B_{+}$ and $B_{-}$.  Often it is advantageous to consider $X_{2n}$ as a subspace of either $B_{+}$ or $B_{-}.$  When we consider $X_{2n}$ as a subspace of a $3$-ball $B^3$ with an $n$--strand trivial tangle $T$ removed (such as $B_{+}$ or $B_-)$, we often discuss the \emph{disk complex} $\mathcal{D}(X_{2n}, B^3 \backslash T)$, which is a subcomplex of $\mathcal{AC}(X_{2n})$ whose vertices are curves which bound disks in $B^3 \backslash T$.  The \emph{Hempel distance of L with respect to $S^2$}, denoted $d(L,S^2)$, is the length of the shortest path in $\mathcal{AC}(X_{2n})$ between $\mathcal{D}(X_{2n}, B_{-})$ and $\mathcal{D}(X_{2n}, B_+)$.  We encourage an interested reader to see \cite{Bachman2003DistanceAB}.

\begin{figure}[ht]
    \centering
    \includegraphics[scale=0.7]{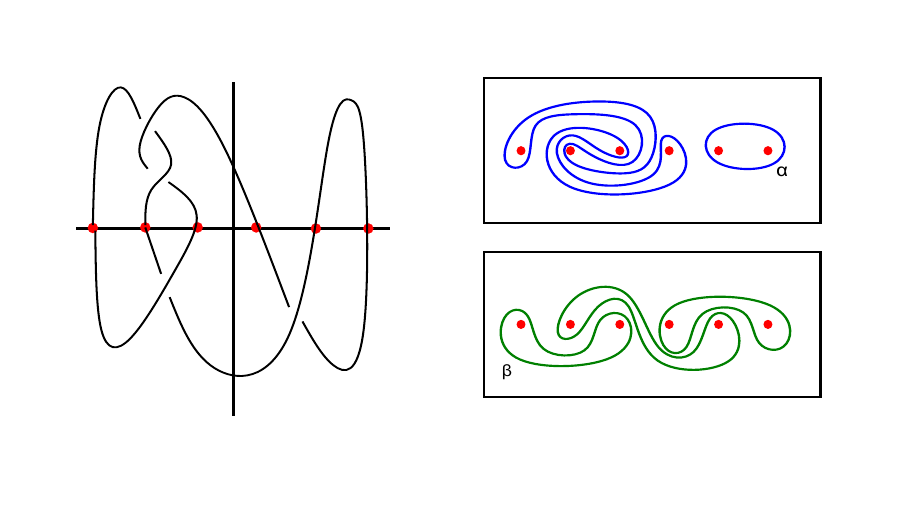}
    \caption{The left figure is a link in bridge position.  The right two figures show the splitting sphere $X_{2n}.$  The top right shows three curves in $\mathcal{D}(X_{2n}, B_+)$ and the bottom right shows three curves in $\mathcal{D}(X_{2n}, B_-).$ Notice that since the curve labeled $\alpha$ is disjoint from the curve labeled $\beta,$ the Hempel distance of this link is 1.}
    \label{fig:hempeldistanceandVs}
\end{figure}

Given a link $L$ in bridge position with respect to a bridge sphere $S$, take a point of $L \cap S$ and perturb a neighborhood of the point slightly to add two new critical points to $L$ (a maximum and a minimum).  The resulting link $L^\prime$ is now in $(n+1)$-bridge position because $S$ is still a bridge sphere for $L^\prime$ and $L^\prime$ now intersects $S$ in $2(n+1)$ points.  This move is called a \emph{stabilization} of $L$.  On the other hand, canceling a maximum and a minimum by an isotopy in a neighborhood of a point of $L$ intersecting the bridge sphere $S$ is called a \emph{destabilization}.  An example is illustrated in Figure~\ref{fig:bridgestabilization}.  If $L$ can be destabilized we say that it \emph{admits a destabilization.}

\begin{figure}[ht]
    \centering
    \includegraphics[scale=0.8]{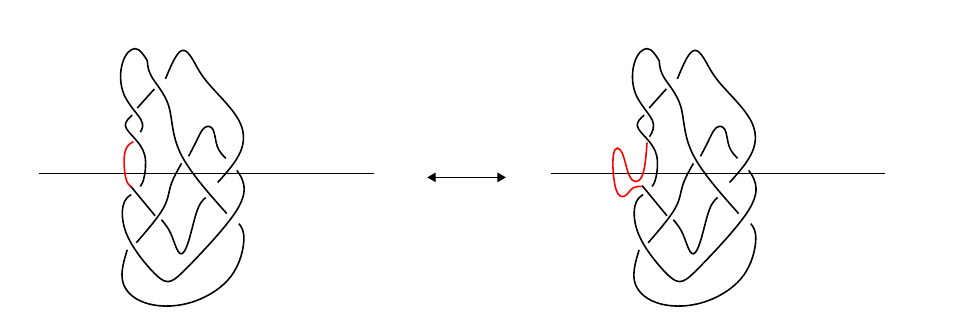}
    \caption{An arbitrary link in bridge position with a stabilization/destabilization occurring along the red subarc}
    \label{fig:bridgestabilization}
\end{figure}

\subsubsection{Known Results about the Hempel Distance of Knots}\label{subsubsec:known_results_hempel_distance}

The Hempel distance of a knot with respect to a splitting sphere provides a lot of information about the knot.  For instance, the following results have been shown in \cite{Bachman2003DistanceAB} and \cite{ozawa_takao_2013}: 
\begin{theorem}\cite{Bachman2003DistanceAB}\cite{ozawa_takao_2013}\label{thm:distance_geq_one_no_destab}
   If a knot $K$ in bridge position has Hempel distance greater than 1, it does not admit a destabilization.
\end{theorem}

\begin{theorem}\cite{ozawa_takao_2013}\label{thm:composite_distance_one}
    Any bridge position of a composite knot has Hempel distance 1.
\end{theorem}
\vspace{1em}

The Hempel distance of a knot also has implications on the types of essential surfaces that may exist in the knot exterior.  For instance,
\begin{theorem}\cite{Bachman2003DistanceAB}\label{thm:Hempel_distance_essential_surface}
    Let $F$ be an orientable essential surface properly embedded in the exterior of a knot $K$ in bridge position. Then the Hempel distance of $K$ is bounded above by twice the genus of $F$ plus $\vert \partial F\vert$.
\end{theorem}
\vspace{1em}

By bounding the genus of essential surfaces in the knot exterior, the authors of \cite{Bachman2003DistanceAB} conclude that for large enough distances, there cannot be any essential tori in the knot complement.  For our purposes, we focus on how Theorem~\ref{thm:Hempel_distance_essential_surface} applies to Seifert surfaces.  The \emph{genus of a knot} is defined to be the smallest genus of any orientable embedded spanning surface for $K$ in $S^3$; therefore, Theorem~\ref{thm:Hempel_distance_essential_surface} implies the following: 

\begin{corollary}\cite{Bachman2003DistanceAB}\label{cor:knot_genus_Hempel_distance}
    Suppose $K$ is a knot in bridge position with $d(K,S^2)=M$.  Then the genus of $K$ is at least $\frac{1}{2}(M-1).$ 
\end{corollary}
\vspace{1em}

This result is very important to the construction in this paper, as a lower bound on the genus of Seifert surfaces will allow us to distinguish distinct knots. 
Also of particular importance to our work is the following result of \cite{Bachman2003DistanceAB}:
\begin{theorem}\cite{Bachman2003DistanceAB}\label{thm:Hempel_distance_hyperbolic_volume}
    If $K$ is a knot whose distance is at least 3 with respect to some bridge sphere $S^2$, then the complement of $K$ is hyperbolic of finite volume.
\end{theorem}
\vspace{1em}


\section{Plats and the Double Branched Cover}\label{sec:Lifts_of Plats}
In this section, we utilize the fact that a link $L$ in plat position is well suited to explicitly construct the double branched cover of $S^3$ over $L$ with a Heegaard splitting inherited from $L$.  We then relate the Hempel distance of $(Y,H)$ and of $(S^3, L)$, proving Proposition~\ref{introprop:knot_distance_bounded_below} and Lemma~\ref{lem:hempel_distances_increasing_beta_and_phi}, which we will use in our construction in Section~\ref{sec:Generating_a_Sequence_of_Hyperbolic_Knots}.

\subsection{Constructing the Double Branched Cover}\label{subsec:Building_the_Double_Branched_Cover}
Given a knot $K=\widehat{\beta}$ for $\beta \in B_{2n}$ and a bridge splitting $(K, X_{2n})$, where $X_{2n}$ is $S^2$ with $2n$ marked points (punctures) where $L \cap S^2$.   We view $S^3$ as $\mathbb{R}^3 \cup \{\infty\}$ and $X_{2n}$ as $\{(x,y,z) | z=0\} \cup \{\infty\}$.  We see that $X_{2n}$ splits $K = \widehat{\beta}$ into $\beta=\beta_1\beta_2$.  Perform a sequence of Reidemeister $2$ moves across $X_{2n}$ so that $\beta = \beta_1 \beta_2 (\beta_2)^{-1} \beta_2$, such that $\beta \cap \{(x,y,z)|z\geq 0\} = \beta_1\beta_2$ and $\beta \cap \{(x,y,z)|z\leq 0\} = \e$, the empty word.  Then $K \cap \{(x,y,z) | z\geq 0\} =T_\beta$ and $K \cap \{(x,y,z) | z \leq 0\} = \overline{T_\e}$.  See Figure~\ref{fig:isotopysotopisnice} for an example. 

\begin{remark}
    The operation in the previous paragraph to obtain $K = T_\beta \cup \overline{T_\e}$ is not strictly speaking necessary for the following construction.  However, it makes calculations significantly easier, as this will translate to the $\alpha$ curves in the Heegaard splitting being standard.  
\end{remark}
\vspace{1em}

\begin{figure}[ht]
    \centering
    \includegraphics[scale=0.8]{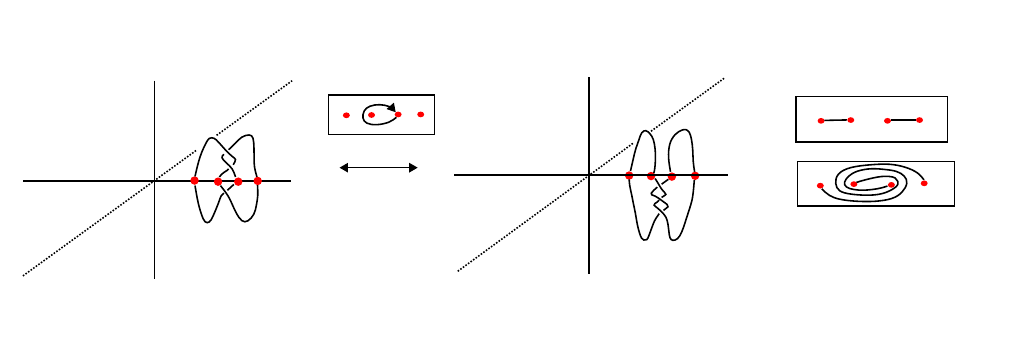}
    \caption{Performing the isotopy indicated untangles the arcs in the upper 3-ball at the expense of tangling further the arcs in the lower 3-ball. The projections of the arcs onto the splitting sphere is the rightmost figure; these are \emph{shadows} as defined below. }
    \label{fig:isotopysotopisnice}
\end{figure}

A \emph{shadow diagram} on $X_{2n}$ is a collection $\Scal=\{a_1, ..., a_n\}$ of properly embedded pairwise disjoint arcs, each of which is called a \emph{shadow}.  Given any trivial tangle $T$, there is a shadow diagram such that $T$ is isotopic on to $\Scal$; in this case, we say that $\Scal$ is a shadow diagram fro $T$.  In general, there are many such shadow diagrams; see a forthcoming paper of the first author for further discussion.  If $\Scal=\{a_i\}_{i=1}^n$ is a shadow diagram for $T=\{t_i\}_{i=1}^n$ such that $t_i$ is isotopic onto $a_i$ for $1\leq i \leq n$, then each pair $t_i, a_i$ cobound an embedded disk in $B_+$ called a \emph{bridge disk}.  
We fix a shadow diagram $\Scal_\e$ for $\overline{T}_\e$ as in Figure~\ref{fig:isotopysotopisnice}.  We choose any shadow diagram $\Scal_\beta$ which $T_\beta$ is isotopic onto.  

Take $S^3= B^3 \cup_{X_{2n}}\overline{B^3}$.  Let $\psi \colon Y(K) \ra S^3$ be the double branched covering map of $S^3$ over $K$.  In $Y(K)$, $X_{2n}$ lifts to the closed surface $F_{n-1}.$  
$B^3_+$ will lift to a genus $n-1$ handlebody $H_+$ in $Y(K)$ with boundary $F_{n-1}$.  Similarly, $B^3_-$ lifts to $H_-$.  Therefore, the genus $0$ splitting $S^3=B^3_+ \cup_{X_{2n}} B^3_-$ induces a Heegaard splitting $Y(K)=H_+\cup_\varphi \overline{H_-}$ where $\varphi \colon F_{n-1} \ra F_{n-1}$ is the splitting map.  We say that $\varphi$ is the splitting map \emph{induced} by $\beta$.  We can use the shadow diagrams $\Scal_0, \Scal_\beta$ to determine diagrammatically what $\varphi$ is. 

\begin{figure}[ht]
    \centering
    \includegraphics[width=0.7\linewidth]{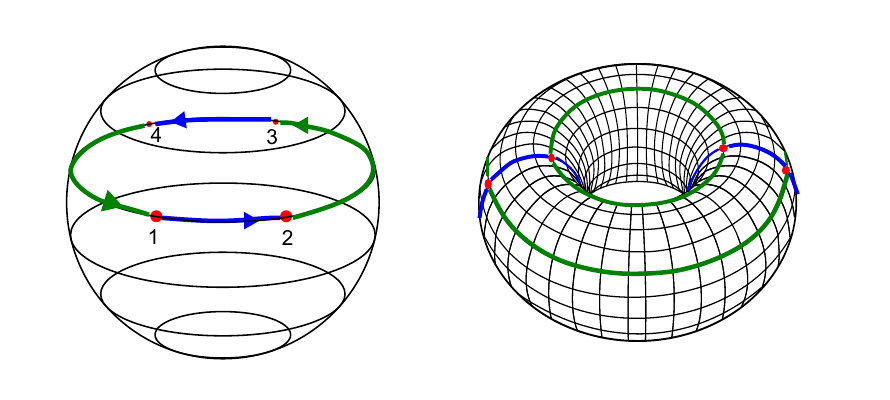}
    \caption{The double branched cover of $X_4$ is a torus.}
    \label{fig:doublebranchedcover}
\end{figure}

Any shadow diagram $\Scal$ on $X_{2n}$ will lift under the double-branched covering map to $n$ nonseparating curves on $F_{n-1}$.   Assume a shadow $a$ has a bridge disk $D_a$ in $B^3_+$ with $\bdry D_a = a \cup t$, where $t$ is a strand in the upper tangle for $K$.  Then $D_a$, lifted to the double branched cover $Y(K)$, becomes a two disks which are glued on $t$ (because $t$ is part of the branched set).  This, $D_a$ lifts to an embedded disk in $H_+$ whose boundary is the curve which is the double of $a$.  
When we apply this to a shadow diagram $\Scal,$ we get $n$ disjoint compressing curves for $H_+$.  It must be that one such curve is homologically a linear combination of the other curves.  From this, we have the following: 

\begin{lemma}\label{lem:shadows_lift_disks}
    Let $K=T_1 \cup \overline{T_2}$ be a knot in $n$ bridge position and $\Scal_1, \Scal_2$ be shadow diagrams for $T_1, T_2$ respectively.  Let $\psi \colon Y(K) \ra S^3$ be the double branched covering map inducing $Y(K) = H_+ \cup_\varphi \overline{H_-}$.  Then $\psi^{-1}(\Scal_1)$ is a set of $n$ compressing curves in $\bdry H_+$ such that any choice of $n-1$ of these curves defines the handlebody $H_+$.  Similarly, $\psi^{-1}(\Scal_2)$ is a set of $n$ compressing curves in $\bdry H_-$ such that any choice of $n-1$ of these curves defines the handlebody $H_-$.
\end{lemma}
\vspace{1em}

Therefore, we can see that $H_+$ is defined by $n-1$ curves in $\psi^{-1}(\Scal_\beta)$ and $H_-$ is defined by $n-1$ curves in $\psi^{-1}(\Scal_0)$.  However, we can see that $\Scal_\beta$ is precisely the image of $\Scal_0$ under the action of $\beta$ on $X_{2n}$.  The trace of this isotopy is a set of bridge disks for $T_\beta$. 

\begin{figure}[ht]
    \centering
    \includegraphics[scale=0.8]{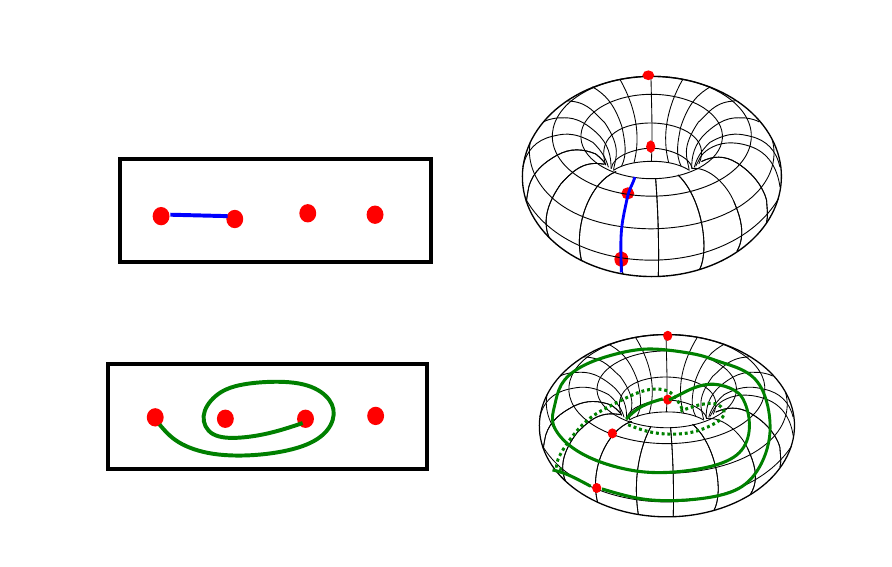}
    \caption{Taking the preimage of one arc from each projection, we obtain two curves on a torus, each bounds a disk.  The green curve is a $(3,1)$-curve.  So, the double branched cover over the trefoil is $L(3,1).$ Compare with Figure \ref{fig:isotopysotopisnice}.}
    \label{fig:SeeingHsplitting}
\end{figure}

\subsection{The Double Branched Cover and Hempel Distance}\label{subsec:Double_Branched_Cover_and_Hempel_Distance}

We take a plat $(K, X_{2n})$ and $\psi \colon Y(K) \ra S^3$ the double branched cover such that $\psi^{-1}(X_{2n})=F_{n-1}$.  

\begin{lemma} \label{lem:arc_curve_complex_embed}
    $\mathcal{AC}(X_{2n})$ embeds into $\mathcal{C}(F_{n-1})$.
\end{lemma} 

\begin{proof}

    The lift of a properly embedded arc in $X_{2n}$ is a nonseparating simple closed curve on $F_{n-1}$.  The preimage of a homologically nontrivial simple closed curve $c$ on $X_{2n}$ is two copies of a curve in $F_{n-1}$.  When $n=2,3$ the lifts of these curves are in fact isotopic \cite{LiftsofcurvesBirmanHilden}.  For $n\geq 4$, the components of $\psi^{-1}(c)$ may not be isotopic, but they will be disjoint (as they do not intersect the branched set).  
     
    Thus, if $v_1, v_2$ are vertices in $\mathcal{AC}(X_{2n})$, then $\psi^{-1}(v_1), \psi^{-1}(v_2)$ are vertices in $\mathcal{C}(F_{n-1})$.     
    Let $v_1, v_2$ be vertices in $\mathcal{AC}(X_{2n})$ which are connected by an edge.  That is, there are representatives of $v_1, v_2$ which are disjoint.  In particular, if $v_1, v_2$ are properly embedded arcs, this implies that $v_1, v_2$ do not share an endpoint.  Then, because the representatives are also disjoint at the branched points, we have that the lifts of their representatives are disjoint.  Therefore if $v_1, v_2$ are connected by an edge, then $\psi^{-1}(v_1), \psi^{-1}(v_2)$ are connected by an edge.  
\end{proof}
\vspace{1em}

\begin{remark}\label{remark:curve_complex_not_isomorphic}
    We note that while $\mathcal{AC}(X_{2n})$ embeds into $\mathcal{C}(F_{n-1})$, they are \emph{not} isomorphic.  Consider a separating curve on $F_2$; this projects to an inessential curve on $X_{2n}$. 
\end{remark}
\vspace{1em}

The lift of a curve in $\mathcal{D}(X_{2n}, B_+)$ or $\mathcal{D}(X_{2n}, B_-)$ (the curves on $X_{2n}$ that bound disks in $B_+$ or $B_-$) is two nonseparating curves, each of which bounds a disk in $H+$ (or $H_-$, respectively).  We extend the definition of the disk complex $\mathcal{D}(F, H)$, where $H$ is a handlebody bounded by a closed surface $F$, to be the set of vertices in $\mathcal{C}(F)$ which bound disks in $H$.  From this and Lemma~\ref{lem:arc_curve_complex_embed}, we have the following.  

\begin{lemma}\label{lem:disk_complex_embed}
    $\mathcal{D}(X_{2n}, B_+)$ embeds into $\mathcal{D}(F_{n-1}, H_+)$.  Similarly, $\mathcal{D}(X_{2n}, B_-)$ embeds into $\mathcal{D}(F_{n-1}, H_-)$.
\end{lemma}

\begin{proof}
   Let $c \in \mathcal{D}(X_{2n}, B_+)$.  Then by Lemma~\ref{lem:disk_complex_embed}, $\psi^{-1}(c)$ is two isotopic copies of a curve or two disjoint curves.  The disk $D$ which $c$ bounds exists away from the branched set $K$; therefore $\psi^{-1}(D)$ is two isotopic or disjoint disks in $H_+$. 
\end{proof}
\vspace{1em}

These results lead us to our main proposition:  

\knotdistanceboundedbelow *

\begin{proof}
    Let $\psi \colon Y(K) \ra S^3$ be the branched covering map.  
    It follows from work by Birman and Hilden in \cite{LiftsofcurvesBirmanHilden} that minimally intersecting arcs and curves in $X_{2n}$ lift to minimally intersecting curves in $F_{n-1}$.  Let $c_1, c_2$ be curves in $\mathcal{D}(X_{2n}, B_+), \mathcal{D}(X_{2n}, B_-)$ such that $d(K,S)=d(c_1, c_2)$.  Then $\psi^{-1}(c_1), \psi^{-1}(c_2)$ are minimally intersecting. 
    Notice, $\psi^{-1}(c_1), \psi^{-1}(c_2)$ cannot be further apart than $c_1, c_2$  (because $\psi$ is a double branched cover).
    Therefore, $d(\psi^{-1}(c_1), \psi^{-1}(c_2)) \leq d(c_1, c_2)$.  We see then that 
    $$d(Y,H) \leq d(\psi^{-1}(c_1), \psi^{-1}(c_2)) \leq d(c_1, c_2) = d(K, S)$$
\end{proof}
We say that two curves $c_1, c_2$ \emph{fill} a surface $F$ if $F \backslash \{c_1, c_2\}$ is a collection of disks, where each disk contains at most one puncture.  It is well known that if a pair of curves fill a surface then in the curve complex of the surface, their vertices  are at least distance 3 apart \cite{primeronmappingclassgroups}.

\begin{lemma}\label{lem:filling_curves}
    For $n \geq 3$, let $c_1, c_2$ be two curves which fill $X_{2n}$, then the lifts of those curves fill $F_{n-1}$.  Therefore, if $c_1, c_2$ fill $X_{2n}$, then $d(\psi^{-1}(c_1), \psi^{-1}(c_2)) \geq 3$.
\end{lemma}



\begin{remark}
   In the case of $n=2$, the double branched covering map induces an isometry of the arc and curve complex of $X_4$ into the curve complex of $F_1.$  It is not hard to see that all genus 1 Heegaard splittings of $S^3$ have Hempel distance 0 (as they are all stabilizations of genus 0 Heegaard splittings).  Therefore, if the plat closure of $\beta\in B_4$ is a knot and its Heegaard splitting has Hempel distance 0, $\beta$ is the unknot. This can easily be computed by obtaining the symplectic representation for the lift of $\beta$ and checking if the bottom left entry of the matrix is a 0.  This explicitly recognizes all 2-bridge unknots. 
\end{remark}
\vspace{1em}
\begin{lemma}\label{lem:beta_pA_phi_pA}
    If the braid word $\beta\in B_{2n}$ is pseudo-Anosov, then the splitting map $\varphi \colon F_{n-1} \ra F_{n-1}$ for the induced Heegaard splitting $H$ of the double branched cover $Y(\widehat{\beta})$ of $S^3$ over $\widehat{\beta}$ is pseudo-Anosov.
\end{lemma}

\begin{proof}
    If a braid word is pseudo-Anosov, it prints two foliations on $X_{2n};$ a stable and unstable foliation.  The lift of these foliations provide stable and unstable foliations of $F_{n-1}.$ In fact, they have the same stretch factor $\lambda$. 
\end{proof}
In particular, the train--track of a pseudo-Anosov braid $\beta$ on $X_{2n}$ will lift to the train--track of $\varphi$.

\subsection{Detecting Generic Pseudo-Anosov Braids via Train-Tracks}\label{subsubsec:Train-Tracks_Method_for_Standard_Pseudo-Anosov_Braids}
Given a closed surface $F$ of genus $g \geq 2$.  In order to construct the knots in Theorem~\ref{introthm:infhypprimeknots}, we will rely on pseudo-Anosov maps $\phi \colon F \ra F$ such that, as splitting maps, the Heegaard splittings determined by powers $\phi^m$ have increasing Hempel distance as $m$ increases.  However, not all pseudo-Anosov maps have this property.  In fact, there exist pseudo-Anosov maps that fix a collection of meridian curves in the handlebody \cite{MMQuasiconvexity}. An explicit example of such an element is given in \cite{pAnotgenericref}.  

    Thankfully, the authors of \cite{pAnotgenericref} also determined that these pathological pseudo-Anosov elements have unstable laminations that converge to meridians.  Moreover, the collection of such elements have measure zero in the the space of all measured laminations.  See \cite{genericpaheegaardsplittings} for a thorough explanation.  Thus, for the vast majority of pseudo-Anosov maps $\phi$, the sequence $\phi^m$ will have increasing Hempel distance as $m$ increases.  

Given a pseudo-Anosov braid $\beta$, we verify that the Hempel distance of $\widehat{\beta^m}$ increases as $m$ increases by considering its train-track.  Take $\beta \in B_{2n}$ pseudo-Anosov and $\widehat{\beta}=T_\beta\cup \overline{T_\e}$ with splitting sphere $X_{2n}$ as in Section~\ref{subsec:Building_the_Double_Branched_Cover}.  
  We say our pseudo-Anosov braid $\beta$ is \emph{generic} if no part of the train track contains a curve or arc on $X_{2n}$ that bound disks in either $B_{+}$ or $B_{-}$.   This involves two steps: (1) computing the train-track associated to $\beta$ and (2) determining if an arc or curve in the train--track extends to a properly embedded disk in either $B_{\pm}.$   The first step is done using an algorithm provided by Bestvina and Handel \cite{MR1308491}.  This algorithm has been implemented by Hall and is freely available \cite{TrainsAlg}. See also \cite{BraidsandDynamics} for a MatLab program implementation.  The second step can be determined for a $c \subset X_{2n}$ curve by using the fact that $c \in \mathcal{D}(X_{2n}, B_\pm)$ if and only if its homology class can be expressed as a sum of $n$ curves that surround the $n$ intersections of the bridges with the sphere $X_{2n},$ that is, curves are the curves on $X_{2n}$ which ``surround'' the punctures $i$ and $i+1$ for $i=1,n-1.$  To determine if an arc bounds a disk, we take the boundary of an annular neighborhood around the arc.  If this curve bounds a disk then necessarily the arc bounds a disk as well (the isotopy of the annular neighborhood of the arc to the arc provides an isotopy of the two disks.)

\begin{lemma}\label{lem:hempel_distances_increasing_beta_and_phi}
    Let $n\geq 3$.  A pseudo-Anosov braid $\beta \in B_{2n}$ has the property that the Hempel distance of $\widehat{\beta^m}$ increases as $m$ increases if and only if the induced splitting map $\varphi$ for the Heegaard splitting $H$ of the double branched cover $Y(\widehat{\beta})$ of $S^3$ over $\widehat{\beta}$ has the property that the Hempel distance of $\varphi^m$ increases as $m$ increases. 
\end{lemma}

\begin{proof}
    This is a direct consequence of Proposition \ref{introprop:knot_distance_bounded_below}.
\end{proof}

\begin{lemma}\label{lem:distanceatleastpower}
   Let $\beta$ be a generic pseudo-Anosov braid word.  Then $d(S^3, \widehat{\beta^m})$ increases at least linearly in $m$.
\end{lemma}
\begin{proof}
     It was shown in \cite{DistancesofHeegaardsplittings} that the Hempel distance of a generic pseudo-Anosov Heegaard splitting grows linearly in the power of the map.  This result and Lemma~\ref{lem:hempel_distances_increasing_beta_and_phi} imply the statement of the lemma.  
\end{proof}

\section{Generating a Sequence of Hyperbolic Knots}\label{sec:Generating_a_Sequence_of_Hyperbolic_Knots}
In order to generate an infinite family of hyperbolic knots, we want a pseudo-Anosov braid word $\beta$ and powers $m_1, m_2, ...$ such that $\widehat{\beta^{m_i}}$ has one component for each $m_i$.  Thus, we will choose a pseudo-Anosov braid $\beta\in B_{2n}$ such that $\widehat{\beta}$ is a knot. Then we will determine the order of the associated permutation $\pi_{\beta}$.  If the order of $\pi(\beta)$ in the symmetric group $S_{2n}$ is $k$ then $\widehat{\beta^{mk+1}}$ will have $1$-component for every $m \in \mathbb{Z}_{\geq 0}$.  By choosing $\beta$ such that $d(\widehat{\beta},S)\geq 3$, we see this sequence will provide infinitely many $n$--bridge hyperbolic knots, as the Hempel distance increases as $m$ increases by Lemma~\ref{lem:distanceatleastpower}.  It is left to show that these knots are prime and that infinitely many of them are distinct.  

\infhypprimeknots*

\begin{proof}
Take $\beta$ as above.  Let $\pi \colon B_{2n} \ra S_{2n}$ be the canonical projection from the braid group to the symmetric group, and let $k$ denote the order of $\pi(\beta)$ in $S_{2n}$.  Fix a splitting sphere $S$ for $\beta$ such that $S$ splits $\beta$ into $T_\beta$ and $T_\e$ as described in Section~\ref{subsec:Building_the_Double_Branched_Cover}.  The splitting sphere $S$ will remain fixed throughout this proof, and for powers of $\beta$, we will append at the end of $\beta$ (such that $S$ splits $\beta^m$ into $T_0$ and the lower half plat closure of $\beta^m$).  

Because $\beta$ is a generic pseudo-Anosov braid, by Lemma~\ref{lem:beta_pA_phi_pA} and Lemma~\ref{lem:hempel_distances_increasing_beta_and_phi} the lift of $\beta$ will be a generic pseudo-Anosov map of the surface $F_{n-1}$.  Thus, there is some $M$ such that for all $m\geq M,$ the Hempel distance of the Heegaard splitting obtained from the lifting $\beta^{mk+1}$ is larger than 3 by Lemma~\ref{lem:distanceatleastpower}. By Proposition~\ref{introprop:knot_distance_bounded_below},   $d(\widehat{\beta^{mk+1}},S)\geq 3.$  By Theorem~\ref{thm:distance_geq_one_no_destab}, $\widehat{\beta^{mk+1}}$ is hyperbolic.  We also know from Section~\ref{subsubsec:number_of_components} that $\widehat{\beta^{mk+1}}$ is a knot for all $m$. By Theorem~\ref{thm:composite_distance_one}, each $\widehat{\beta^{mk+1}}$ is a prime knot.

By Corollary~\ref{cor:knot_genus_Hempel_distance}, the knot genus $g(\widehat{\beta^{mk+1}})$ is bounded below by $\frac{1}{2}(d(\widehat{\beta^{mk+1}},S)-1).$  This provides a lower bound of the genus of the sequence of knots $\{\widehat{\beta^{mk+1}}\}_{m=M+1}^\infty$.  To exhibit an upper bound on the genus we use Seifert's algorithm on the knot diagram.  For any knot $K \in \{\widehat{\beta^{mk+1}}\}_{m=M+1}^\infty$, if $g(K)=G$, then there exists some $m$ such that $G < \frac{1}{2}(d(\widehat{\beta^{mk+1}},S)-1)$.  Therefore, $\{\widehat{\beta^{mk+1}}\}_{m=M+1}^\infty$ contains infinitely many distinct knots. 
\end{proof}

We note that the upper bound on genus given by Seifert's algorithm is a poor bound; experimental data supports that each $\beta^{mk+1}$ may in fact be unique.  See Example~\ref{example:3-bridgeconcreteexample} and Conjecture~\ref{conj:entropy_genus_distance_volume}.  
%
%
 %
%


\begin{example}\label{example:3-bridgeconcreteexample}
         Consider the braid on 6 strands $\beta=\sigma_2^2\sigma_4\sigma_1\sigma_3\sigma_5\sigma_2.$  Let $S=X_{2n}$, our splitting sphere, be positioned as in Section~\ref{subsec:Building_the_Double_Branched_Cover}.  This braid is pseudo-Anosov, thus its powers are also pseudo-Anosov.  To check that $\beta$ is a \emph{generic} pseudo-Anosov braid, we compute the train-track associated to $\beta.$  We will use Toby Hall's implementation of the Bestvina-Handel algorithm. See \cite{MR1308491} and \cite{TrainsAlg}. 
        \begin{figure}[ht]
            \centering
            \includegraphics[scale=0.5]{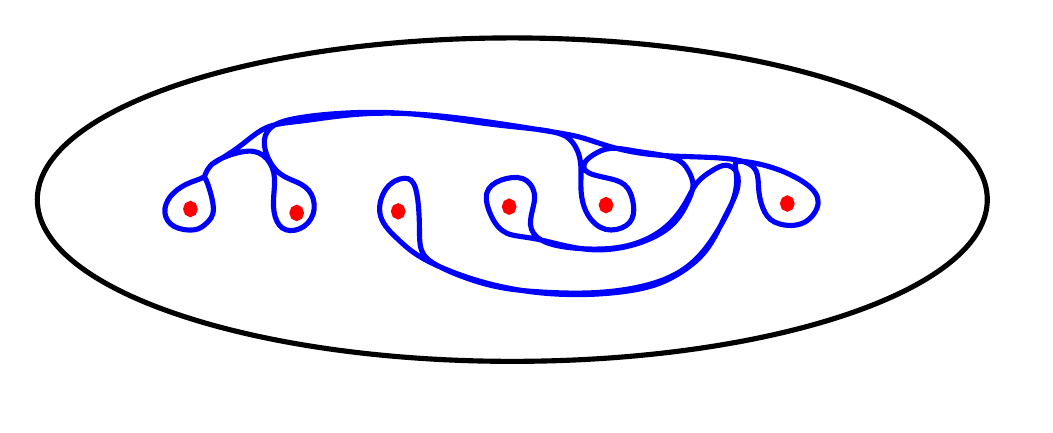}
            \caption{The train-track of $\beta$ on punctured disk}
            \label{fig:traintrackforyourexample}
        \end{figure}
            From Lemma~\ref{lem:beta_pA_phi_pA}, the lift of $\beta$ gives a pseudo-Anosov gluing map.  Therefore, taking the double branched cover over $X_{2n}$ gives a genus 2 surface with the following train-track corresponding to the lift of $\beta$.
        \begin{figure}[ht]
                 \centering
                 \includegraphics[scale=0.5]{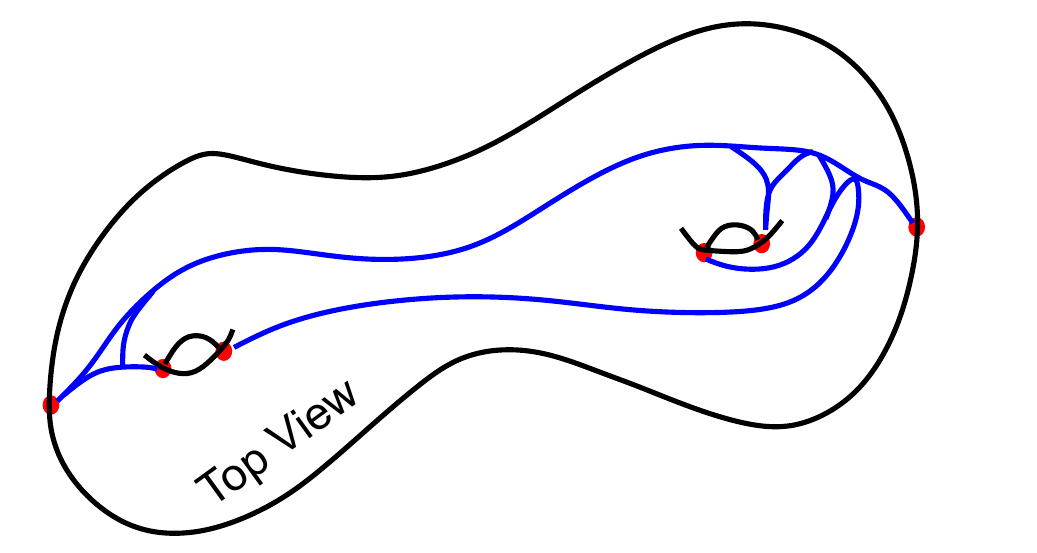}
                 \caption{Train-track of the lift of $\beta.$ Only the top part of the train track is shown; flipping the surface over would show the same image on the bottom side.}
                 \label{fig:traintrackonlift}
        \end{figure}
        There is no subgraph of the track on the disk that is an arc connecting two curves surrounding punctures.  The part of the track joining the curves around punctures 1 and 2 has a valence 3 point on it, as does the part of the track joining curves around punctures 3 and 4.  The lift of this train track then does not contain a subgraph that is a meridional curve. Thus, the lift of $\beta$ and therefore $\beta$ are generic pseudo-Anosov elements.
              
        The associated permutation of $\beta$ is $\pi_{\beta}=[2,5,1,3,6,4]$.  We see by constructing the plat closure graph (or by drawing the plat itself, as in Figure~\ref{fig:pabeta}) that $\widehat{\beta}$ is a knot. 
        The order of the permutation $\pi_{\beta}$ is 6.  Therefore, $\{\widehat{\beta^7}, ...\}$ are knots. 
        Additionally, by constructing the plat closure graph for $\beta$, we see that $\widehat{\beta^5}$ also results in a knot.  Therefore, taking powers of $\beta$ of the form $6m\pm 1$ will always give a knot.

            \begin{figure}[ht]
            \centering
            \begin{tikzpicture}[
                roundnode/.style={circle, draw=green!60, fill=green!5, very thick, minimum size=7mm}]
               \pic[braid/.cd,
             number of strands=6,
            line width=2pt,
            name prefix=braid, 
            height=.25in,
            width=0.25in,
            rotate=180,
            style=thin] at (0,0) {braid={s_4^{-1} s_4^{-1} s_2^{-1} s_5^{-1} s_3^{-1} s_1^{-1} s_4^{-1}}};
            \node[fill=red,circle,inner sep=0pt,minimum size=0pt] at (-1.9, 0)  (t1) {};
            \node[fill=red,circle,inner sep=0pt,minimum size=0pt] at (-1.3, 0)  (t2){};
            \node[fill=red,circle,inner sep=0pt,minimum size=0pt] at (-.6, 0)  (t3){};
            \node[fill=red,circle,inner sep=0pt,minimum size=0pt] at (0, 0)  (t4){};
            \node[fill=red,circle,inner sep=0pt,minimum size=0pt] at (-3.2, 0)  (t5){};
            \node[fill=red,circle,inner sep=0pt,minimum size=0pt] at (-2.5, 0)  (t6){};
            \node[fill=red,circle,inner sep=0pt,minimum size=0pt] at (-1.9, -4.95)  (b1){};
            \node[fill=red,circle,inner sep=0pt,minimum size=0pt] at (-1.3, -4.95)  (b2){};
            \node[fill=red,circle,inner sep=0pt,minimum size=0pt] at (-.6, -4.95)  (b3){};
            \node[fill=red,circle,inner sep=0pt,minimum size=0pt] at (0, -4.95)  (b4){};
            \node[fill=red,circle,inner sep=0pt,minimum size=0pt] at (-3.2, -4.95)  (b5){};
            \node[fill=red,circle,inner sep=0pt,minimum size=0pt] at (-2.5, -4.95)  (b6){};
            \draw[black, bend left]  (t1) to node [auto] {} (t2);
            \draw[black, bend left]  (t5) to node [auto] {} (t6);
             \draw[black, bend left]  (t3) to node [auto] {} (t4);
             \draw[black, bend right]  (b1) to node [auto] {} (b2);
             \draw[black, bend right]  (b3) to node [auto] {} (b4);
             \draw[black, bend right]  (b5) to node [auto] {} (b6);
            \end{tikzpicture}
            \caption{$\beta$ in plat position}
            \label{fig:pabeta}
        \end{figure}
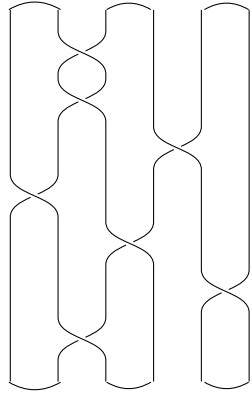
        
        Therefore, taking the plat closure of powers of $\beta$ and applying Theorem~\ref{introthm:infhypprimeknots} generates an infinite sequence of prime hyperbolic knots, infinitely many of which are distinct. In Table~\ref{table:first_example} we provide a table showing the relationship between powers of $\beta$, the entropy, the knot genus, and the hyperbolic volume.  This data was calculated in Sage and SnapPy. 
\begin{table}[ht] 
\caption{The entropy, genus, and volume for $\widehat{\beta^n}$ when $\beta=\sigma_2^2\sigma_4\sigma_1\sigma_3\sigma_5\sigma_2$}
\centering
\begin{tabular}{lllll}
 $n$& Entropy & Genus&  Hyperbolic volume \\ \hline
 1& 1.4860 & 0  & n/a   \\
 5& 7.43 & 4  &  16.075  \\
 7& 10.402 & 6  & 38.5061  \\
 11& 16.346 & n/a  &  76.12916  \\
 13& 19.318 & n/a  & 94.8582   \\
 17& 25.262 & n/a  & 132.4251  \\
 19& 28.234 & n/a  & 151.20236   \\
 23& 34.178 & n/a  & 188.761   \\
 25& 37.15 & n/a  & 204.3726  \\
\end{tabular}
\label{table:first_example}
\end{table}

\end{example}

\subsection{Relationship to Highly Twisted Links}\label{subsec:highly_twisted}

In \cite{bridgedistanceplat}, Johnson and Moriah consider consider plats which have the form shown in Figure~\ref{fig:highlytwisted} below, which we will refer to as \emph{fishnet plats}. Each of the boxes is called a \emph{twist box} and is denoted by $a_{i,j}$ for $1\leq i\leq n-1$ and $1\leq j\leq m.$  The integer $n$ is referred to as the \emph{height} of the plat, and $m$ is called the \emph{width} (in our work, $m$ is the bridge number of the plat).  Each twist box is assigned an integer that denotes how many half twists occur in the box, and whether they are right-handed (positive) or left-handed (negative).  Such a diagram is \emph{highly twisted} if $|a_{i,j}| \geq 3$ for every $a_{i,j}$.

\begin{figure}[ht]
    \centering
    \includegraphics[height=6cm]{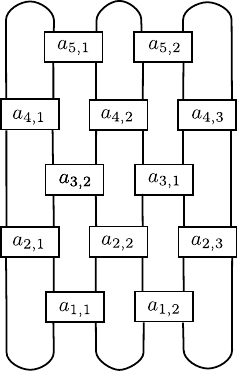} \hspace{2cm} \includegraphics[height=6cm]{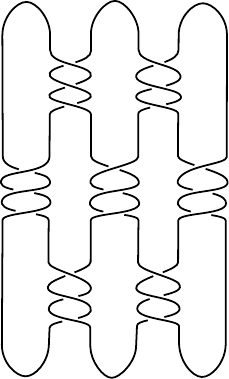}
    \caption{On the left is a general diagram for a height 6 width 3 fishnet plat, as described in \cite{bridgedistanceplat}.  On the right is an example of a highly twisted plat with $a_{1,j}=a_{3,j}=3$ for $j=1,2$ and $a_{2,j}=-3$ for $j=1,2,3$.}
    \label{fig:highlytwisted}
\end{figure}

While every link admits a presentation as a fishnet plat, not every link admits a presentation as a highly twisted plat.  For highly twisted plats, Johnson and Moriah prove the following: 

\begin{theorem}[Theorem 1.2 of \cite{bridgedistanceplat}]\label{theorem:bridgedistanceplats}
If $K\subseteq S^3$ is an $m$--bridge link with a highly twisted $n$--row $2m$--plat projection, with $m\geq 3,$ then the distance of the Hempel distance of the plat is exactly $\left\lceil \frac{n}{2(m-2)}\right\rceil.$
\end{theorem}

\noindent In \cite{HighlyTwisted}, Lazarovich, Moriah, and Pinksy expand upon this result to knots and links which are not necessarily in plat position, showing:

\begin{theorem}[Theorem A of \cite{HighlyTwisted}]\label{thm:highly_twisted_A}
    Let $D(\mathcal{L})$ be a connected, prime, twist-reduced, $3$-highly twisted diagram of a link $\mathcal{L}$ with at least two twist regions, then $\mathcal{L}$ is hyperbolic. 
\end{theorem}


There is a deep connection between hyperbolic knots and knots which admit a highly-twisted diagram.  However, when in plat position, the restriction to highly-twisted braid words is very strict; braid words which are highly twisted are of a very specific form.  
In our Example~\ref{example:3-bridgeconcreteexample}, $\beta$ is \emph{not} a highly-twisted diagram, nor are powers of $\beta$.  However, this does not preclude the possibility that these knots do admit some highly-twisted diagram.

\begin{example}\label{example:highlytwisted}


While our construction does not require highly-twisted braid words, it does interact very nicely with them.  For example, we consider the braid word: $\beta=\sigma_2^3\sigma_4^3\sigma_1^{-3}\sigma_3^{-3}\sigma_5^{-3}\sigma_2^3\sigma_4^3$ as in Figure~\ref{fig:highlytwisted}.   
First, we show this braid word is pseudo-Anosov by considering its lift to $Y(\widehat{\beta})$.  Lifting $\sigma_2$ and $\sigma_4$ to their corresponding action on the genus 2 splitting surface for $Y(\widehat{\beta})$, we see that they correspond to a multi-twist along the multicurve containing the two longitudinal curves.  Similarly, the letters $\sigma_1,\sigma_3,\sigma_5$ lift to a multi-twist along the multicurve containing the two meridional curves and a curve that intersects each longitude exactly once.  This is illustrated in Figure~\ref{fig:humpriesgenerators}.  Since the union of the two multicurves fill the surface, it follows from work of Penner \cite{PennerPAconstruction} that taking positive powers of twists along one multicurve and negative powers of twists along the other multicurve results in a pseudo-Anosov map on the genus 2 surface.  Subsequently $\beta$ is a pseudo-Anosov map on the punctured sphere since $\beta$ and its lift have the same local properties, i.e. the same stable and unstable foliations in a small neighborhood.  (For more on Penner's construction, see Theorem 14.4 of \cite{primeronmappingclassgroups}.)  Therefore, $\beta$ is a pseudo-Anosov braid word. 

Further, we see that $\beta$ must be generic in the sense of Section~\ref{subsubsec:Train-Tracks_Method_for_Standard_Pseudo-Anosov_Braids} because its plat closure is highly twisted; all subsequent powers of $\beta$ will result in highly twisted plats, and by Theorem~\ref{theorem:bridgedistanceplats} the Hempel distance of these powers must then increase linearly.

\begin{figure}[ht]
    \centering
    \includegraphics[width=.4\linewidth]{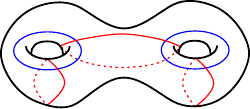} 
    \caption{The multicurve corresponding to lifts of $\sigma_2$ and $\sigma_4$ in blue. 
 The multicurve corresponding to lifts of $\sigma_1,\sigma_3,\sigma_5$ in red.}
    \label{fig:humpriesgenerators}
\end{figure}

We chose our braid $\beta$ specifically so that it would be both highly twisted and fit the hypotheses of Theorem~\ref{introthm:infhypprimeknots}.  The associated permutation of $\beta$ is $\pi(\beta)=[3,5,1,6,2,4].$  We see by constructing the plat closure graph that $\beta$ is a knot.  The order of the permutation $\pi(\beta)$ is 2.  Therefore, $\{\widehat{\beta^{2k+1}}\vert k=0,1,2,\dots \}$ is a collection of knots.  Using Theorem~\ref{theorem:bridgedistanceplats} we can calculate explicitly the Hempel distances of these plats:
\begin{align*}
    d(\widehat{\beta^{2k+1}})&=\left\lceil \frac{n}{(2(m-2)}\right\rceil  \text{\hspace{1em} for all $k,$ $m=3$ }\\
    &=\left\lceil \frac{n}{2}\right\rceil  \text{\hspace{1em} for each $k,$ $n=4(k+1)$ }\\
    &=\left\lceil \frac{4(k+1)}{2}\right\rceil = 2(k+1) \\
\end{align*}
This demonstrates that the Hempel distance grows linearly with powers of $\beta$ as stated in Lemma~\ref{lem:distanceatleastpower}.
In Table~\ref{table:second_example} we provide a table showing the first few knotted powers of $\beta,$ the entropy, the Hempel distance, and hyperbolic volume. Note that we were unable to calculate the knot genus for powers of this knot due to computational limitations. This data was calculated in Sage and SnapPy.
\begin{table}[ht]
\caption{The entropy, genus, and volume for $\widehat{\beta^n}$ when $\beta=\sigma_2^3\sigma_4^3\sigma_1^{-3}\sigma_3^{-3}\sigma_5^{-3}\sigma_2^3\sigma_4^3$}
\centering
\begin{tabular}{lllll}
 $k$& Entropy & Hempel distance &  Hyperbolic volume \\ \hline
 0& 4.02503 & 2  & 28.75 \\
 1& 12.07510 & 4  & 90.49  \\
 2& 20.12515 & 6  & 152.20 \\
 3& 28.17521 & 8   & 213.92 \\
 4& 36.2257 & 10  & 275.63 \\
 5& 44.27533 & 12  & 337.35 \\
 6& 52.32539 & 14  & 399.06  \\
 7& 60.37545 & 16  & 460.78  \\
 8& 68.42551 & 18  & 522.49  \\
\end{tabular}
\label{table:second_example}
\end{table}
\end{example}

\begin{example}\label{example:3_-1}
   By the same argument as in Example~\ref{example:highlytwisted}, the braid word $\beta=\sigma_2^3\sigma_4^3\sigma_1^{-1}\sigma_3^{-1}\sigma_5^{-1}\sigma_2^3\sigma_4^3$ is a pseudo-Anosov braid.  By lifting its train-track, we see that it is generic in the sense of Section~\ref{subsubsec:Train-Tracks_Method_for_Standard_Pseudo-Anosov_Braids}.  The permutation $\pi_\beta = [3, 5, 1, 6, 2, 4]$ is as in Example~\ref{example:highlytwisted}, so $\{\widehat{\beta^{2k+1}}\}$ is an infinite family of prime hyperbolic knots as in Theorem~\ref{introthm:infhypprimeknots}.  However, $\widehat{\beta^{2k+1}}$ are not highly twisted plats, as so we cannot explicitly calculate the Hempel distance.  
\end{example}

\begin{example}\label{example:why_n_geq_3}
    Now we present a non-example which illustrates why the construction above relies on bridge index $n\geq 3$.  The results about the Hempel distance for knots rely on $n \geq 3$ and the results about Hempel distance for Heegaard splittings rely on $g \geq 2$ largely because the curve complexes for $X_{2n}$, $n \leq 2$ and $F_g$, $g \leq 1$ are very different than those for $n \geq, g \geq 2$ \cite{primeronmappingclassgroups}.  
    In this example, we see an pseudo-Anosov braid $\beta \in B_4$ where $\widehat{\beta^m}$ are in fact of the same knot type.

    Consider the braid word $\beta=\sigma_2\sigma_1^{-1}\sigma_2^2\sigma_3\in B_4.$  This is a pseudo-Anosov braid whose plat closure is the figure--eight knot.  The corresponding permutation $\pi_{\beta}=[3,1,4,2]$ has order $4$. This means taking the plat closure of $\beta^{2m-1}$ yields a knot.  On the other hand, by directly checking the knot type for the first few $m$ shows something strange: the knot type does not change! By a simple inductive argument, we show that the knot type of powers $\beta^{2m-1}$ will always result in a figure--eight knot.
    \vspace{1em}\\
    For $m=1,2,3$ by direct calculation, we have that the plat closure of $\beta^{2m-1}$ is the figure--eight knot.  Suppose now that the plat closure of $\beta^{2m-1}$ for an arbitrary $m$ is the figure--eight knot.  Consider the plat closure of $\beta^{2(m+1)-1}=\beta^{2m-1+2}.$  By assumption, we have that the plat closure of $\beta^{2m-1}$ is the figure--eight knot, which is the same as the plat closure of $\beta.$ 
    So that the knot type of $\beta^{2m-1+2}$ is the same as $\beta^{3}$, which is already known to be the figure--eight.
\end{example}

\section{Questions and Further Work}\label{sec:Questions and Further Work}

We end with several open questions and conjectures. As stated above, it seems likely that the Seifert genus is a poor upper bound for the genus of the knots in $\{\widehat{\beta^{mk+1}}\}_{m=M+1}^\infty$.  Indeed, if $F$ is a Seifert surface for $\widehat{\beta}$, one can obtain a spanning surface for $\widehat{\beta^2}$ by attaching $n$ untwisted bands to the surface from the upper bridges of one copy of $\beta$ to the lower bridges of the other.  However, there is no guarantee that this operation produces an orientable surface.  If one attempts to rectify this by adding half twists to the bands, then the boundary is no longer $\widehat{\beta^2}$.

Recall that the \emph{entropy} of a pseudo-Anosov braid is defined as the log of the stretch factor of the map.  See \cite{BraidsandDynamics} for more discussion.  The hyperbolic volume of a knot is linearly related to its genus \cite{brittenham1998boundingcanonicalgenusbounds}, which is linearly related to Hempel distance by Corollary~\ref{cor:knot_genus_Hempel_distance}.  The Hempel distance is linearly related to the power of the map, the log of the stretch factor (entropy) then is linearly related to the power of the map. Therefore, we conjecture these are all linearly related to the power of the pseudo-Anosov braid.

\begin{conjecture}\label{conj:entropy_genus_distance_volume}
    Let $\beta\in B_{2n}$ with $n\geq 3$ be a genetic pseudo-Anosov braid.  Then the entropy of braids in the sequence $\{\beta^{mk+1}\}_{m=M+1}^\infty$, a lower bound on the knot genus of $\{\widehat{\beta^{mk+1}}\}_{m=M+1}^\infty$, the Hempel distances of $\{(\widehat{\beta^{mk+1}},S)\}_{m=M+1}^\infty$, and the hyperbolic volumes of the $\{\widehat{\beta^{mk+1}}\}_{m=M+1}^\infty$ are all linearly related by a function in $m$.
\end{conjecture}
\vspace{1em}

We have demonstrated that for generic pseudo-Anosov braids with $1$-component plat closure with Hempel distance at least $3$, we have the infinite sequence of Theorem~\ref{introthm:infhypprimeknots}.  The assumption that $d(\widehat{\beta}, S) \geq 3$ is crucial; otherwise, our knot may not be hyperbolic.  Indeed, the pseudo-Anosov braid in Example \ref{example:3-bridgeconcreteexample} an unknot!  This leads to the following question: 

\begin{question}
    Can every link type be represented as the plat closure of a pseudo-Anosov braid? 
\end{question}
\vspace{1em}


The toolkit we've outlined in this paper may also have other wide-reaching applications.  There are deep relationships between links in bridge position and knotted surfaces; see \cite{MZ4mnflds} for more details.  The connections between tangles, shadows, and complexes related to a punctured sphere are further explored in a forthcoming paper from the first author.  Additionally, Hempel distance has $4$-dimensional analogs; see \cite{KT} and \cite{pants} for more.

\bibliographystyle{alpha}
\bibliography{bibliography1}
\vfill

\newcommand{\Addresses}{{
    Carolyn Engelhardt, \textsc{Department of Mathematics, University at Buffalo-SUNY,
    Buffalo, NY 14260-2900, USA}\par\nopagebreak
    \textit{E-mail address}: \texttt{cengelha@buffalo.edu}
    \vspace{1em}\\
    Seth Hovland, \textsc{Department of Mathematics, University at Buffalo-SUNY,
    Buffalo, NY 14260-2900, USA}\par\nopagebreak
  \textit{E-mail address}: \texttt{sethhovl@buffalo.edu}

}}

\Addresses

\end{document}